%% file: SAHS.tex
\documentclass[10pt,reqno]{amsart}

\usepackage{SAHS}

\input{s/frontmatter}

\begin{document}
\maketitle

\begin{abstract}
\input{s/abstract}
\end{abstract}

\input{s/intro}
\input{s/SectionSCT}
\input{s/SectionSCH}
\input{s/SectionFLT}
\input{s/SectionHB}

\bibliography{SAHS}
\bibliographystyle{amsplain}
\end{document}

%% file: s/frontmatter.tex
\title{A supercharacter approach to Heilbronn sums}

\author{Stephan Ramon Garcia}
	\address{Department of Mathematics, Pomona College, 610 N. College Ave., Claremont, CA 91711} 
	\email{Stephan.Garcia@pomona.edu}
	\urladdr{\url{http://pages.pomona.edu/~sg064747}}

\author{Bob Lutz}
    \address{Department of Mathematics, University of Michigan,
2074 East Hall,
530 Church Street,
Ann Arbor, MI  48109-1043}
    \email{boblutz@umich.edu}

\thanks{Partially supported by NSF Grant DMS-1265973.}

%% file: s/abstract.tex
Various algebraic properties of Heilbronn's exponential sum can be deduced 
through the use of supercharacter theory, a novel extension of classical
character theory due to Diaconis-Isaacs and Andr\'e.  This perspective yields
a variety of formulas and provides a method for computing the number of
solutions to Fermat-type congruences.

%% file: s/intro.tex
\section{Introduction}
	The theory of \emph{supercharacters}, of which classical character theory is a special case, was introduced 
	by P.~Diaconis and I.M.~Isaacs in 2008 \cite{Diaconis}, generalizing the \emph{basic characters} studied by
	C.~Andr\'e \cite{An95, An01, An02}.  
	The original aim of supercharacter theory was to provide new tools for studying groups,
	such as the unipotent matrix groups $U_n(q)$, that had proven intractable from the perspective of classical
	character theory.  However, recent work indicates that supercharacters on abelian groups
	are intimately tied to various exponential sums arising in the theory of numbers \cite{RSS, SESUP, GNGP, GNSG, gausscyclotomy, Kloosterman}.
	Our aim here is to explore another such connection, demonstrating that many standard properties of Heilbronn's exponential sum 
	can be systematically deduced through supercharacter theory.

	We adopt the standard notation $e(x) = \exp(2 \pi i x)$, so that the function $e(x)$ is periodic with period $1$.
	The letter $p$ will always denote an odd prime number and $g$ a primitive root modulo $p^2$.
	In this note, we show that \emph{Heilbronn sums}\footnote{The notation is not completely standardized.  For instance,
		Heath-Brown denotes the sum running from $1$ to $p$ by $S(a,p)$ in \cite{HB2}, whereas Heath-Brown and Konyagin use
		$H_p(a)$ to denote the sum running from $1$ to $p$ in \cite{HBK}.  We adopt here the notation used in Kowalski's lecture notes
		on exponential sums \cite{Kowalski}, which invites less confusion with Kloosterman or Sal\'ie sums and is also more suitable from
		the viewpoint of supercharacter theory.}
	\begin{equation}\label{eq-Heilbronn}
		H_p(a) = \sum_{\ell=1}^{p-1} e\left( \frac{a\ell^p}{p^2} \right)
	\end{equation}	
	arise as the values of 
	\emph{supercharacters} on $\Z/p^2\Z$ induced by the action of a certain subgroup of the unit group
	$(\Z/p^2\Z)^{\times}$.  This observation, coupled with the general techniques from \cite{RSS, SESUP}, permit us
	to derive a variety of identities involving Heilbronn sums.  The novelty of our approach lies in the use of
	supercharacter theory, which reduces many computations to matrix arithmetic.

	A brief review of basic facts about supercharacters on abelian groups is undertaken in Section \ref{SectionSCT},
	after which we construct the relevant supercharacter theory for Heilbronn sums in Section \ref{SectionSCH}.
	An exact formula
	involving Heilbronn sums for computing the number of solutions to 
	Fermat-type congruences $ax^p + by^p \equiv cz^p \pmod{p^2}$ is given 
	in Section \ref{SectionFLT}.
	We conclude in Section \ref{SectionHB} with an exact formula for quartic sums
	involving Heilbronn sums.

%% file: s/SectionSCT.tex
\section{Supercharacters on abelian groups}\label{SectionSCT}	

	Before proceeding, we recall a few basic facts about supercharacters on 
	abelian groups.  Since complete details can be found in \cite{RSS, SESUP}, we
	content ourselves with a quick overview of the relevant facts required in our particular case.

	Let $A$ be a subgroup of $GL_d( \Z/n\Z)$ that is closed under the transpose operation and
	let $X_1,X_2,\ldots,X_N$ denote the orbits in $G=(\Z/n\Z)^d$ under the action of $A$.  The functions
	\begin{equation}\label{eq-Supercharacter}
		\sigma_i(\vec{y}) = \sum_{\vec{x} \in X_i} e \left( \frac{ \vec{x} \cdot \vec{y} } {n} \right),
	\end{equation}
	where $\vec{x}\cdot \vec{y}$ denotes the formal dot product of two elements of $(\Z/n\Z)^d$,
	are called \emph{supercharacters} on $(\Z/n\Z)^d$ and the sets $X_i$
	are referred to as \emph{superclasses}.  It turns out that supercharacters are constant
	on superclasses, and hence we may employ the notation $\sigma_i(X_j)$ without
	confusion.  The $N \times N$ matrix	
	\begin{equation}\label{eq-USCT}
		U = \frac{1}{\sqrt{n^d}} \left[  \frac{   \sigma_i(X_j) \sqrt{  |X_j| }}{ \sqrt{|X_i|}} \right]_{i,j=1}^N
	\end{equation}
	is symmetric (i.e., $U = U^T$) and unitary.  In fact, the matrix $U$ encodes an analogue of discrete Fourier transform (DFT)
	on the space of all \emph{superclass functions} (i.e., functions $f:(\Z/n\Z)^d\to\C$ that are constant on each superclass)
	and satisfies many of the standard properties of the DFT \cite{SESUP}.
More general supercharacter theories on certain abelian groups are studied in
	\cite{Hendrickson, HendricksonNew, HendricksonThesis}.		
	
	It turns out that a variety of exponential sums that are relevant to the theory of numbers can be realized
	as supercharacters on abelian groups in the manner described above.  This approach was
	first undertaken to study Ramanujan sums \cite{RSS} and, a short while later, Gaussian periods
	\cite{GNGP,gausscyclotomy}.  The general theory is developed in \cite{SESUP}, where a number of such examples
	(see Table \ref{TableNT}) are discussed.  A novel and visually compelling class of exponential sums is considered from 
	the supercharacter perspective in \cite{GNSG}.
	
	\input{t/TableNT}	
	
	The main result we require is the following, which identifies the set of all matrices that are diagonalized
	by the unitary matrix \eqref{eq-USCT} as the span of a certain family of matrices containing combinatorial information about
	the superclasses.  A complete proof and further details can be found in \cite{SESUP} (see also \cite{RSS}).

	\begin{Lemma}\label{LemmaSESUP}
		Let $A = A^T$ be a subgroup of $GL_d(\Z/n\Z)$, let $\X = \{X_1,X_2,\ldots,X_N\}$
		denote the set of superclasses induced by the action of $A$ on $(\Z/n\Z)^d$, and let $\sigma_1,\sigma_2,\ldots,\sigma_N$
		denote the corresponding supercharacters.
		For each fixed $z$ in $X_k$, let $c_{i,j,k}$ denote the number of solutions $(x_i,y_j) \in X_i \times X_j$ to the equation $x+y = z$.
		\begin{enumerate}\addtolength{\itemsep}{0.5\baselineskip}
			\item $c_{i,j,k}$ is independent of the representative $z$ in $X_k$ which is chosen,
			\item The identity
				\begin{equation}\label{eq-sijk}
					\sigma_i(X_{\ell}) \sigma_j(X_{\ell}) = \sum_{k=1}^N c_{i,j,k} \sigma_k(X_{\ell})
				\end{equation}			
				holds for $1\leq i,j,k,\ell \leq N$.
			\item The matrices $T_1,T_2,\ldots,T_N$, whose entries are given by
				\begin{equation}\label{eq-Tijk}
					[T_i]_{j,k} = \frac{ c_{i,j,k} \sqrt{ |X_k| } }{ \sqrt{ |X_j|} },
				\end{equation}
				each satisfy 
				\begin{equation}\label{eq-TUUD}
					T_i U = U D_i,
				\end{equation}
				where 
				\begin{equation}\label{eq-DM}
					D_i = \operatorname{diag}\big(\sigma_i(X_1), \sigma_i(X_2),\ldots, \sigma_i(X_N) \big).
				\end{equation}
				In particular, the $T_i$ are simultaneously unitarily diagonalizable.
				
			\item Each $T_i$ is a normal matrix (i.e., $T_i^*T_i = T_i T_i^*$) and
				the set $\{T_1,T_2,\ldots,T_N\}$ forms a basis for the algebra of all $N \times N$ matrices 
				$T$ such that $U^*TU$ is diagonal.
		\end{enumerate}
  	\end{Lemma}

%% file: t/TableNT.tex
	\begin{table}
		\begin{equation*}\footnotesize
		\begin{array}{|c|c|c|c|}
			\hline
			\text{Name} & \text{Expression} & G & A \\
			\hline\hline
			\text{Gauss} & 
			\eta_j = \displaystyle \sum_{\ell=0}^{d-1} e\left( \frac{g^{k\ell+j}}{p}\right)
			& \Z/p\Z & \text{nonzero $k$th powers mod $p$} \\[20pt]
			\text{Ramanujan} & c_n(x)=\displaystyle \sum_{ \substack{ j = 1 \\ (j,n) = 1} }^n \!\!\!\! 
			e\left(\frac{ jx}{n} \right) & \Z/n\Z & (\Z/n\Z)^{\times} \\[20pt]
			\text{Kloosterman} & K_p(a,b)=\displaystyle \sum_{ \ell = 0  }^{p-1} e\left( \frac{a\ell + b \overline{\ell} }{p}\right)
			 & (\Z/p\Z)^2 & \left\{ \minimatrix{u}{0}{0}{u^{-1}}  : u \in (\Z/p\Z)^{\times} \right\} \\[20pt]
			\text{Heilbronn} & \displaystyle H_p(a)=\sum_{\ell=0}^{p-1} e\left(\frac{a \ell^p}{p^2} \right) & \Z/p^2\Z & 
			\footnotesize\text{nonzero $p$th powers mod $p^2$} \\[20pt]
			\hline
		\end{array}
		\end{equation*}	
		\caption{\footnotesize Gaussian periods, Ramanujan sums, Kloosterman sums, and Heilbronn sums 
		appear as supercharacters arising from the action of a group $A$ of automorphisms on an
		abelian group $G$.  Here $p$ denotes an odd prime number and $g$ a primitive root modulo $p$.}
		\label{TableNT}
	\end{table}	

%% file: s/SectionSCH.tex
\section{A supercharacter theory for Heilbronn sums}\label{SectionSCH}

	We are now in a position to represent Heilbronn sums as the values of certain supercharacters on $(\Z/n\Z)^d$, where $d=1$ and $n=p^2$ for an odd prime $p$.
	Following the general outline described in Section \ref{SectionSCT}, we first require
	a group of automorphisms $A$ to act upon $G = \Z/p^2\Z$.  To this end, we need the following
	lemma.
	
	\begin{Lemma}\label{LemmaSquare}
		If $p$ is an odd prime, then for all integers $x$ and $y$ we have $x^p \equiv y^p \pmod{p^2}$ if and only if $x \equiv y \pmod{p}$.
	\end{Lemma}
	
	\begin{proof}
		If $x \equiv y \pmod{p}$, then $x = y+rp$ for some integer $r$. Therefore
		\begin{equation*}
			x^p = (y+rp)^p = \sum_{k=0}^p \binom{p}{k} y^{p-k} (rp)^k \equiv y^p \pmod{p^2},
		\end{equation*}
		since $p|\binom{p}{k}$ for $k=1,2,\ldots,p-1$. 
		
		Suppose now that $x^p \equiv y^p \pmod{p^2}$. Let
		$g$ be a primitive root modulo $p^2$, and write $x \equiv g^j$ and $y \equiv g^k\pmod{p^2}$. We see that
		$g^{jp} \equiv g^{kp} \pmod{p^2}$, whence $g^{(j-k)p} \equiv 1 \pmod{p^2}$. It follows that $(j-k)p$ is a multiple of $\phi(p^2) = p(p-1)$,
		so $(p-1)|(j-k)$.  Writing
		$j = k + m(p-1)$, we see that
		\begin{equation*}
			x \equiv g^j \equiv g^{k+m(p-1)} \equiv g^k (g^{p-1})^m \equiv g^k \equiv y, \pmod{p}
		\end{equation*}
		by Fermat's little theorem.
	\end{proof}
	
	It follows that
	\begin{equation}\label{eq-A}
		A = \{ 1^p , 2^p,\ldots, (p-1)^p\}
	\end{equation}
	is a subgroup of $(\Z/p^2\Z)^{\times}$ of order $p-1$.  
	Letting $A$ act upon $\Z/p^2\Z$ by multiplication, we obtain the orbits
	\begin{align*}
		X_1 &= gA,\\
		X_2 &= g^2A,\\
		&\vdots\\
		X_{p-1} &= g^{p-1}A,\\
		X_p &= A,\\
		X_{p+1} &= \{p,2p,\ldots,(p-1)p\},\\
		X_{p+2} &= \{0\},
	\end{align*}
	where $g$ denotes a primitive root modulo $p^2$ that will remain fixed throughout this paper.  
	We have adopted this somewhat unusual labeling scheme in order to simplify the structure of certain matrices and
	streamline a number of formulas which appear later.	
	For $1 \leq i,j \leq p$, we find that
	\begin{equation*}
		\sigma_i(X_j) = \sum_{\ell=1}^{p-1} e\left( \frac{g^{j} (g^{i} \ell^p)}{p^2}\right)
		= \sum_{\ell=1}^{p-1} e\left( \frac{g^{i+j}  \ell^p}{p^2}\right)
		= H_p(g^{i+j}).
	\end{equation*}
	Additionally, since $A$ is closed under negation and $e(-z)=\overline{e(z)}$ for all $z\in \C$, all Heilbronn sums are real.
	We pause to make the following observation.
	
	\begin{Lemma}\label{LemmaDOUP}
		The value of $H_p(g^k)$ depends only upon $k \pmod{p}$.
	\end{Lemma}
	
	\begin{proof}
		Since $(g,p)=1$, the map $\ell\mapsto g^j\ell$ is a permutation of $\Z/p\Z$
		for each $j$.  In light of Lemma \ref{LemmaSquare} we conclude that
		\begin{equation*}
			H_p(g^{k+jp})
			=\sum_{\ell=1}^{p-1} e\left(\frac{g^{k+jp}\ell^{p}}{p^2}\right) 
			= \sum_{\ell=1}^{p-1} e\left(\frac{g^k(g^j\ell)^{p}}{p^2}\right)
			=\sum_{r=1}^{p-1} e\left(\frac{g^kr^{p}}{p^2}\right)=H_p(g^k),
		\end{equation*}
		as desired.
	\end{proof}

	Upon performing some additional elementary computations to evaluate the remaining values of $\sigma_i(X_j)$,
	we obtain the \emph{supercharacter table} corresponding to the supercharacter theory on $\Z/p^2\Z$
	arising from the action of $A$ (see Table \ref{TableSCT}).  
	\input{t/TableSCT}
	Also of relevance is the unitary matrix $U$ defined by \eqref{eq-USCT}, which is given by
	\begin{equation}\label{eq-U}
		U=\small\frac{1}{p}\left[
		\begin{array}{ccccc|cc}
			 H_p(g^2)& H_p(g^3) & H_p(g^4) & \cdots  & H_p(g)  & -1 & \sqrt{p-1}\\
			 H_p(g^3)& H_p(g^4) & H_p(g^5) & \cdots & H_p(g^2)  & -1 & \sqrt{p-1}\\
			 H_p(g^4)& H_p(g^5) & H_p(g^6) & \cdots & H_p(g^3)  & -1 & \sqrt{p-1}\\
			 \vdots & \vdots & \vdots  & \iddots & \vdots & \vdots &\vdots \\
			 H_p(g) & H_p(g^2) & H_p(g^3)  &\cdots & H_p(1)& -1 & \sqrt{p-1}\\\hline
			-1 & -1 & -1 & \cdots & -1 & p-1 &\sqrt{p-1} \\
			\sqrt{p-1} & \sqrt{p-1} & \sqrt{p-1} & \cdots & \sqrt{p-1} & \sqrt{p-1}& 1\\
		\end{array}
		\right].
	\end{equation}

	We obtain the following identities from the fact that \eqref{eq-U} is unitary:
	\begin{align}
		\sum_{\ell=1}^p H_p(g^\ell) &= 0, \label{eq-Sum}\\
		\sum_{\ell=1}^p H_p^2(g^\ell) &= p(p-1), \label{eq-Norm}  \\
		\sum_{\ell=1}^p H_p(g^\ell)H_p(g^{i+\ell}) &= -p. \label{eq-DotProduct}  \
	\end{align}
	Identity \eqref{eq-Sum} is obtained by taking the inner product of the first column of $U$ with the $(p+1)$st.
	Identity \eqref{eq-Norm} is obtained by noting that the first column of $U$ has unit norm. Identity \eqref{eq-DotProduct} is obtained by taking the inner product of any two columns of $U$ among the first $p$ columns.
	Squaring \eqref{eq-Sum}, expanding, and using \eqref{eq-Norm} provides us with
	\begin{equation}\label{eq-snt}
		\sum_{1\leq r<s\leq p} H_p(g^r) H_p(g^s) = -\frac{p(p-1)}{2}.
	\end{equation}

	In light of the fact that $U$ is a real symmetric unitary matrix, we see that $U^2= I$
	whence the only possible eigenvalues of $U$ are $\pm 1$.  In fact, we can say much more.

	\begin{Proposition}
		The matrix $U$ has eigenvalues $1$ and $-1$ with multiplicities $(p+3)/2$ and $(p+1)/2$, respectively.
		In particular,
		\begin{equation*}
			\det U = 
			\begin{cases}
				-1 & \text{if $p \equiv 1 \pmod{4}$}, \\
				1 & \text{if $p \equiv 3 \pmod{4}$}.
			\end{cases}
		\end{equation*}
	\end{Proposition}

	\begin{proof}
		Since the only possible eigenvalues are $\pm 1$, it suffices to show that $\tr U = 1$.
		In light of Lemma \ref{LemmaDOUP} and the fact that $p$ is odd, it follows that
		\begin{equation*}
			\tr U 
			= \frac{1}{p} \Big( \sum_{\ell=0}^{p-1} H_p(g^{2\ell}) + (p-1)+1\Big) 
			= 1+ \frac{1}{p} \sum_{\ell=0}^{p-1} H_p(g^{\ell})   
			= 1,  
		\end{equation*}
		by \eqref{eq-Sum}.  Thus $1$ and $-1$ have the multiplicities claimed.
	\end{proof}

	\begin{Proposition}
		The upper-left $p \times p$ matrix 
		\begin{equation}\label{eq-HMatrix}
			H=
			\frac{1}{p}
			\begin{bmatrix}
				 H_p(1)& H_p(g) & H_p(g^2) & \cdots  & H_p(g^{p-1})  \\
				 H_p(g)& H_p(g^2) & H_p(g^3) & \cdots & H_p(1)  \\
				 H_p(g^2)& H_p(g^3) & H_p(g^4) & \cdots & H_p(g)  \\
				 \vdots & \vdots & \vdots  & \iddots & \vdots  \\
				 H_p(g^{p-1}) & H_p(1) & H_p(g)  &\cdots & H_p(g^{p-2}) \\
			\end{bmatrix}
		\end{equation}
		has the eigenvalues $0,1,-1$ with multiplicities $1$, $\frac{p-1}{2}$, $\frac{p-1}{2}$,
		respectively.
	\end{Proposition}
	
	\begin{proof}
		It follows immediately from \eqref{eq-Norm} and \eqref{eq-DotProduct} that
		\begin{equation}\label{eq-HH}
			H^2
			= \frac{1}{p}
			\begin{bmatrix}
				p-1 & -1 & -1 & \cdots & -1 \\
				-1 & p-1 & -1 & \cdots & -1 \\
				-1 & -1 & p-1 & \cdots & -1 \\
				\vdots & \vdots & \vdots & \ddots & \vdots \\
				-1 & -1 & -1 & \cdots & p-1 
			\end{bmatrix}
			 =  I - \frac{1}{p}\vec{u}\vec{u}^T ,
		\end{equation}
		where $\vec{u}$ denotes the $p \times 1$ vector consisting of all ones.  Thus
		\begin{equation*}
			H^3 = H(I - \frac{1}{p}\vec{u}\vec{u}^T ) = H + \frac{1}{p}(H\vec{u}) \vec{u}^T = H
		\end{equation*}
		since $H\vec{u} = \vec{0}$ by \eqref{eq-Sum}.  Since $H^3 = H$, we conclude that the eigenvalues
		of $H$ are among $0$, $1$, and $-1$.  In light of \eqref{eq-HH}, it follows that $\tr H^2 = p-1$ whence $H$ has precisely
		$p-1$ nonzero eigenvalues.  Since $\tr H = 0$ by \eqref{eq-Sum}, we next see that $H$
		has the eigenvalues $0,1,-1$ with multiplicities $1$, $\frac{p-1}{2}$, $\frac{p-1}{2}$,
		respectively.
	\end{proof}

	We now identify the matrices $T_i$ from
	Lemma \ref{LemmaSESUP}.  Following the recipe developed there, we let $c_{i,j,k}$ denote 
	the number of solutions $(x,y)$ in $X_i\times X_j$ to 
	\begin{equation}\label{eq-cijk}
		x+y\equiv z \pmod{p^2},
	\end{equation}
	where $z$ is a fixed element of $X_k$, recalling that the value of $c_{i,j,k}$ is independent
	of the representative $z$ of $X_k$.  Lemma \ref{LemmaSESUP} ensures that $U$ simultaneously
	diagonalizes the matrices $T_1,T_2,\ldots,T_{p+2}$ whose entries are given by \eqref{eq-Tijk}.
	To be more specific, we have $T_iU=UD_i$, where 
	\begin{equation}\label{eq-D}
		D_i=\operatorname{diag}\big(\sigma_i(X_1),\sigma_i(X_2),\ldots,\sigma_i(X_{p+2})\big).
	\end{equation}
	Since each eigenvalue $\sigma_i(X_j)$ is real and $U$ is unitary, it follows that each $T_i$ is real and symmetric.
	In order to describe the matrices $T_1,T_2,\ldots,T_p$, we first
	require a few elementary facts about the $c_{i,j,k}$.

	\begin{Lemma}\label{LemmaPermute}
		If $1\leq i,j,k\leq p$, then $c_{i,j,k}=c_{\pi(i,j,k)}$ for any permutation $\pi(i,j,k)$.
	\end{Lemma}
	
	\begin{proof}
		The identity $c_{i,j,k}=c_{j,i,k}$ is immediate.
		For any $z\in X_k$ the solutions $(x,y)\in X_i\times X_j$ of $x+y\equiv z\pmod{p^2}$ are also the solutions of $wx-wz\equiv -g^j\pmod{p^2}$, obtained by rearranging terms and multiplying through by $w=g^jy^{-1}$, where the inverse $y^{-1}$ is taken modulo $p^2$. As the pair $(x,y)$ ranges over $X_i\times X_j$, the pair $(wx,-wz)$ ranges over $X_i\times X_k$ since $w$ is a $p$th power modulo $p^2$ and $X_k$ is closed under negation. Hence $c_{i,j,k}=c_{i,k,j}$. The result follows.
	\end{proof}

	\begin{Lemma}\label{LemmaRowSum}
		If $1 \leq i\leq p$ and $j\neq i$, then 
		\begin{equation}\label{eq-RowSum}
			\sum_{k=1}^{p+2} c_{i,j,k}=p-1.
		\end{equation}
	\end{Lemma}

	\begin{proof}
		We first note that if $1 \leq i\leq p$ and $j\neq i$, then
		$c_{i,j,p+2} = 0$,
		since $a^p g^i + b^p g^j \equiv 0 \pmod{p^2}$ has no solutions when $i \neq j$.
		Indeed, the preceding is equivalent to $a^p g^{i-j} \equiv (-b)^p \pmod{p^2}$,
		which is inconsistent because $g^{i-j} A \cap A = \varnothing$.  Since $c_{i,j,p+2} = 0$ and
		$(\Z/p^2\Z)\backslash\{0\} = X_1 \cup X_2 \cup \cdots \cup X_{p+1}$,
		it follows that any sum of the form $x+y$, where $x$ and $y$ belong to $X_i$ and $X_j$, respectively,
		also belongs to $(\Z/p^2\Z)\backslash\{0\}$.  For $1 \leq k \leq p-1$, the superclass $X_k$
		has precisely $p-1$ distinct representatives whence $x+y$ belongs to $X_k$ for precisely
		$(p-1)c_{i,j,k}$ pairs $(x,y)$ in $X_i \times X_j$.  Thus $(p-1)^2= |X_i\times X_j|=\sum_{k=1}^{p+1} (p-1)c_{i,j,k}$,
		which implies \eqref{eq-RowSum}.
	\end{proof}
	
	We now have all of the information required to describe the general structure of $T_1,T_2,\ldots,T_p$.

	\begin{Lemma}\label{LemmaT}
		If $1 \leq i\leq p$, then
		\begin{equation}\label{eq-Ti}
			T_i =\small
			\left[
			\begin{array}{ccccccc|cc}
		                &    &   &  &  &  &  & 1 & 0\\
		                   &   &  &  &  &  &  & \vdots & \vdots \\
		                  &   &    & & & & & 1 &0\\
		                  & &  & C_i & & & &0 &\sqrt{p-1}\\
				  &  & & & & & & 1 & 0 \\
				 &  & & & &  & &\vdots &\vdots\\
				  &  & & & & & & 1&0\\
				  \hline
		                 1  & \cdots & 1 & 0 & 1 &\cdots &1 &0 &0\\
				 0 & \cdots & 0 & \sqrt{p-1} & 0 & \cdots & 0 & 0&0
			\end{array}
			\right],
		\end{equation}
		where $C_i = [c_{i,j,k}]_{j,k=1}^p$ and the $\sqrt{p-1}$ occurs in the $ith$ row and $i$th column. 
	\end{Lemma}

	\begin{proof}
		Suppose that $1 \leq i \leq p$.  Since $T_i$ is real and symmetric, it suffices to establish that the final two columns
		of $T_i$ have the desired form.
		In what follows, $a$ and $b$ denote units modulo $p$.
		
		We first show that the upper-right $p \times 2$ submatrix is of the form claimed.
		Let us consider the coefficients $c_{i,j,p+1}$ for $j \neq i$.  Since
		\begin{equation*}
			a^p g^i + b^p g^j \equiv p \pmod{p^2} 
			\quad\iff\quad
			ag^i + bg^j \equiv 0 \pmod{p},
		\end{equation*}
		for each fixed $a$ we may let $b \equiv -ag^{i-j} \pmod{p}$ to obtain a solution to the preceding congruences.
		In particular, this implies that $c_{i,j,p+1} \geq 1$ for $j \neq i$.	
		However, Lemmas \ref{LemmaPermute} and \ref{LemmaRowSum} tell us that
		$\sum_{j=1}^{p+2} c_{i,j,p+1} = p-1$,
		from which it follows that 
		\begin{equation}\label{eq-cijp1}
			c_{i,j,p+1} = 
			\begin{cases}
				1 & \text{if $j \neq i$},\\
				0 & \text{if $j= i$},
			\end{cases}
		\end{equation}
		as claimed.  Turning our attention to the final column of $T_i$, we note that
		the proof of Lemma \ref{LemmaRowSum} tells us that $c_{i,j,p+2} = 0$ for $j \neq i$.
		Moreover, $c_{i,i,p+2} = p-1$
		for $1 \leq i \leq p$ since
		\begin{equation*}
			a^p g^i + b^p g^i=g^i(a^p + b^p) \equiv 0 \pmod{p^2}
		\end{equation*}
		has exactly $p-1$ solutions $\{ (a,-a) : 1 \leq a \leq p-1\}$.  In other words,
		\begin{equation}\label{eq-ciip2}
			c_{i,j,p+2} = 
			\begin{cases}
				0 & \text{if $j \neq i$},\\
				p-1 & \text{if $j= i$}.
			\end{cases}
		\end{equation}
		
		That the lower-right $2 \times 2$ submatrix of $T_i$ is identically zero
		follows easily from the fact that $a^pg^i$ is a unit modulo $p^2$.
	\end{proof}

Our final lemma will be useful in Section \ref{SectionHB}.

	\begin{Lemma}\label{LemmaSRS}
		For $1 \leq i \leq p$,
		\begin{equation}
			\sum_{k=1}^{p} c_{i,i,k} = p-2. \label{eq-SRS}
		\end{equation}
	\end{Lemma}
	
	\begin{proof}
		First observe that there are exactly $(p-1)^2$ pairs $(a,b)$ with $1 \leq a,b \leq p-1$.
		Since $c_{i,i,k}$ is independent of the representative from $X_k$ which is chosen, it follows that 
		as $(x,y)$ ranges over $X_i\times X_i$, the sum $x+y$ assumes values in $X_k$
		exactly $|X_k| c_{i,i,k}$ times.  In light of \eqref{eq-cijp1} and \eqref{eq-ciip2}, we obtain
		\begin{equation*}
			(p-1)^2 
			= \sum_{k=1}^{p+2} |X_k| c_{i,i,k} 
			= \sum_{k=1}^p (p-1) c_{i,i,k}  +0+ (p-1),
		\end{equation*}
		which implies \eqref{eq-SRS}.
	\end{proof}

%% file: t/TableSCT.tex
\begin{table}\small
		\begin{equation*}
			\begin{array}{|c || cc cc|cc |}
				\hline & X_1 & X_2 & \cdots & X_p & X_{p+1} & X_{p+2} \\ 
				|X_i| & p-1 & p-1 & \cdots & p-1 & p-1 & 1 \\
				\hline\hline
				\sigma_{1} & H_p(g^2) & H_p(g^3) & \cdots & H_p(g) & -1 & p-1\\
				\sigma_{2} & H_p(g^3) & H_p(g^4) & \cdots & H_p(g^2) & -1 & p-1\\
				\vdots & \vdots & \vdots & \iddots & \vdots & -1 & p-1\\
				\sigma_{p} & H_p(g) & H_p(g^2) & \cdots & H_p(1) & -1 & p-1\\ \hline
				\sigma_{p+1} & -1 & -1 & \cdots & -1 & p-1 & p-1\\
				\sigma_{p+2} & 1 & 1 & \cdots & 1 & 1  & 1\\ \hline
			\end{array}
		\end{equation*}
		
		\caption{\footnotesize The supercharacter table corresponding to the supercharacter theory on $\Z/p^2\Z$
		arising from the action of the subgroup $A=\{ 1^p , 2^p,\ldots, (p-1)^p\}$.}
		\label{TableSCT}
	\end{table}	

%% file: s/SectionFLT.tex
\section{The third moment and Fermat's Last Theorem}\label{SectionFLT}
	Although it is not obvious from their definition, Heilbronn's exponential sums
	are related to a certain family of congruences connected to Fermat's Last Theorem.
	Since the details and history of Fermat's Last Theorem are well-known, we make no attempt to 
	discuss the topic in depth, recalling only that this famous conjecture (proved by Andrew Wiles \cite{Wiles}),
	asserts that the equation
	$x^n + y^n = z^n$ has no integral solutions $x,y,z\geq 1$ if $n \geq 3$.  Moreover, the general
	case can be easily reduced to the consideration of odd prime exponents.
	
	\begin{Theorem}\label{TheoremFermat}
		If $p \nmid abc$, then the number of solutions $(x,y,z)$ in $(\Z/p^2\Z)^3$ 
		to the generalized Fermat congruence
		\begin{equation}\label{eq-GFC}
			ax^p + b y^p \equiv c z^p \pmod{p^2}
		\end{equation}
		which satisfy $p\nmid xyz$ is precisely 
		\begin{equation}\label{eq-F1}
			p^3(p-1) F(p;a,b,c), 
		\end{equation}
		where $F(p;a,b,c)$ denotes the nonnegative integer
		\begin{equation}\label{eq-F2}
			 F(p;a,b,c) = 1 - \frac{2}{p} + \frac{1}{p^2} \sum_{\ell=1}^p H_p(ag^{\ell})H_p(bg^{\ell})H_p(cg^{\ell}),
		\end{equation}
		where $g$ denotes a primitive root modulo $p^2$.
		In particular, the equation 
		\begin{equation*}
			ax^p + by^p = cz^p
		\end{equation*}
		has no solutions in integers with $p \nmid xyz$ whenever $F(p;a,b,c) = 0$.
	\end{Theorem}

	\begin{proof}
		If $p\nmid abc$ then $a,b,c$ are congruent modulo $p^2$ to some powers $g^i,g^j,g^k$ of $g$.  
		We may assume without loss of generality that 
		$1\leq i,j,k \leq p$ since $g^{i+\ell p} x^p \equiv g^i (g^{\ell} x)^p \pmod{p^2}$ and so forth. 
		
		Recall that $c_{i,j,k}$ denotes the number of solutions to the congruence 
		\begin{equation*}
			g^i x^p + g^j y^p \equiv g^k \pmod{p^2}
		\end{equation*}
		with $1 \leq x,y \leq p-1$.  Since there are $p-1$ different representatives of the superclass $X_k = g^k A$, there
		are $(p-1)c_{i,j,k}$ solutions to \eqref{eq-GFC} with $1\leq x,y,z \leq p-1$.  By considering
		$(x+rp,y+sp,z+tp)$ for $0 \leq r,s,t \leq p-1$ we obtain 
                $p^3(p-1)c_{i,j,k}$ distinct solutions to \eqref{eq-GFC}.
                
                We will be done if we can show that $c_{i,j,k}$ is equal to the right-hand side of \eqref{eq-F2}. Compute the $(j,k)$ entry of the matrix identity $T_i = UD_iU$ to obtain
        \begin{align}
        	c_{i,j,k} 
        	&= \frac{\sqrt{|X_j|}}{p^2\sqrt{|X_k|}} \sum_{\ell=1}^{p+2} \sigma_i(X_\ell)\sigma_j(X_\ell)\sigma_\ell(X_k) \nonumber\\
        	&=\frac{1}{p^2}\sum_{\ell=1}^{p+2} \sigma_i(X_{\ell})\sigma_j(X_{\ell})\sigma_{\ell}(X_k)\nonumber\\
        	&=\frac{1}{p^2|X_k|}\sum_{\ell=1}^{p+2} |X_{\ell}|\sigma_i(X_{\ell})\sigma_j(X_{\ell})\sigma_k(X_{\ell})\label{eq:4.11}\\
        	&=1 - \frac{2}{p} + \frac{1}{p^2} \sum_{\ell=1}^p H_p(ag^{\ell})H_p(bg^{\ell})H_p(cg^{\ell})\label{eq:4.12},
        \end{align}
        as desired, where \eqref{eq:4.11} follows from the fact that $U = U^T$ and \eqref{eq:4.12} follows from Table \ref{TableSCT}.
	\end{proof}

	As the preceding theorem illustrates, cubic sums of Heilbronn sums control, in a precise manner, whether
	the generalized Fermat congruence \eqref{eq-GFC} possesses any nontrivial solutions.  Indeed,
	we consider a solution satisfying $p|xyz$ trivial since if, say $p|x$, 
	the congruence reduces to $by^p \equiv c z^p \pmod{p^2}$, which has no solutions if
	$b = g^i$ and $c = g^j$ for $i \not\equiv j \pmod{p}$ and has only the $p(p-1)$ obvious solutions
	otherwise.  
	
	We remark that \eqref{eq-F2} can be used to efficiently 
	evaluate $F(p;a,b,c)$ for many triples $(a,b,c)$ in succession.  In fact, one can compute $F(p;a,b,c)$
	for all triples $(a,b,c)$ simultaneously by taking advantage of the identity $T_i = UD_iU$ and fast
	matrix multiplication.

	We present in Table \ref{TableFermat} numerical values of the function $F(p) =F(p;1,1,1)$, which corresponds to the
	classical Fermat congruence $x^p+y^p \equiv z^p \pmod{p^2}$.  In particular, $F(p) = 0$ implies that the Fermat equation
	$x^p + y^p = z^p$ has no solutions in integers satisfying $p \nmid xyz$.

	At this point it is worth mentioning Kummer's proof of Fermat's Last Theorem for regular primes.	
	Recall that a prime $p$ is called \emph{regular} if $p$ does not divide the class number of the cyclotomic field
	$\mathbb{Q}(\zeta)$ where $\zeta = e(\frac{1}{p})$.  It is well-known that $p$ is regular if and only if $p$ does not divide
	the numerator of the Bernoulli numbers $B_2,B_4,\ldots,B_{p-3}$  \cite[p.~198]{StewartTall}.
	Although Kummer himself believed that there are infinitely many regular primes, this conjecture remains open.
	On the other hand, Jensen proved that there are infinitely many \emph{irregular primes} (i.e., primes which are not regular),
	the first few of which are 
	\begin{quote}
		37, 59, 67, {\bf 101}, 103, {\bf 131}, {\bf 149}, 157, {\bf 233}, {\bf 257}, {\bf 263}, 271, 283, {\bf 293}, 307, {\bf 311}, {\bf 347}, {\bf 353}, 
		379, {\bf 389}, {\bf 401}, 409, 421, 
		433, {\bf 461}, 463, 467, {\bf 491}, 523, 541, 547, {\bf 557}, 577, {\bf 587}, {\bf 593}, 607, 613, {\bf 617}, 619, 631, {\bf 647}, {\bf 653}, 
		659, 673, {\bf 677}, 
		{\bf 683}, 691, 727, 751, 757, {\bf 761}, {\bf 773}, {\bf 797}, {\bf 809}, 811, {\bf 821}, {\bf 827}, {\bf 839}, 877, {\bf 881}, 887, 929, {\bf 953}, 971.
	\end{quote}
	The first major step in Kummer's approach is establishing
	that if $p$ is an odd regular prime, then $x^p+y^p = z^p$ has no integral solutions with $p \nmid xyz$ \cite{StewartTall}
	(in the terminology of \cite{BS}, this is referred to as the \emph{first case} of Fermat's Last Theorem).	
	A glance at Table \ref{TableFermat} reveals we have actually established that the Fermat equation $x^p + y^p = z^p$ has 
	no integral solutions $x,y,z\geq 1$ with $p \nmid xyz$ if $p$ is one of the irregular primes highlighted in boldface above.

	\input{t/TableFermat}

%% file: t/TableFermat.tex
	\begin{table}
		\begin{equation*}
		\footnotesize
		\begin{array}{|c|c||c|c||c|c||c|c||c|c||c|c|}
			\hline
			p & F(p) & p & F(p) & p & F(p) & p & F(p) & p & F(p) & p & F(p) \\
			\hline
			 3 & 0 & 127 & 2 & 281 & 0 & 461 & 0 & 647 & 0 & 853 & 2 \\
			 5 & 0 & 131 & 0 & 283 & 2 & 463 & 2 & 653 & 0 & 857 & 6 \\
			 7 & 2 & 137 & 0 & 293 & 0 & 467 & 0 & 659 & 0 & 859 & 2 \\
			 11 & 0 & 139 & 2 & 307 & 2 & 479 & 0 & 661 & 2 & 863 & 0 \\
			 13 & 2 & 149 & 0 & 311 & 0 & 487 & 2 & 673 & 2 & 877 & 2 \\
			 17 & 0 & 151 & 2 & 313 & 2 & 491 & 0 & 677 & 0 & 881 & 0 \\
			 19 & 2 & 157 & 2 & 317 & 0 & 499 & 2 & 683 & 0 & 883 & 2 \\
			 23 & 0 & 163 & 2 & 331 & 2 & 503 & 0 & 691 & 8 & 887 & 6 \\
			 29 & 0 & 167 & 0 & 337 & 8 & 509 & 0 & 701 & 12 & 907 & 8 \\
			 31 & 2 & 173 & 0 & 347 & 0 & 521 & 0 & 709 & 2 & 911 & 6 \\
			 37 & 2 & 179 & 6 & 349 & 2 & 523 & 2 & 719 & 0 & 919 & 2 \\
			 41 & 0 & 181 & 2 & 353 & 0 & 541 & 2 & 727 & 2 & 929 & 6 \\
			 43 & 2 & 191 & 0 & 359 & 0 & 547 & 8 & 733 & 2 & 937 & 2 \\
			 47 & 0 & 193 & 8 & 367 & 2 & 557 & 0 & 739 & 2 & 941 & 0 \\
			 53 & 0 & 197 & 0 & 373 & 2 & 563 & 0 & 743 & 0 & 947 & 0 \\
			 59 & 12 & 199 & 2 & 379 & 2 & 569 & 0 & 751 & 2 & 953 & 0 \\
			 61 & 2 & 211 & 2 & 383 & 0 & 571 & 2 & 757 & 8 & 967 & 2 \\
			 67 & 2 & 223 & 2 & 389 & 0 & 577 & 2 & 761 & 0 & 971 & 6 \\
			 71 & 0 & 227 & 6 & 397 & 2 & 587 & 0 & 769 & 2 & 977 & 6 \\
			 73 & 2 & 229 & 2 & 401 & 0 & 593 & 0 & 773 & 0 & 983 & 0 \\
			 79 & 8 & 233 & 0 & 409 & 2 & 599 & 0 & 787 & 8 & 991 & 2 \\
			 83 & 6 & 239 & 0 & 419 & 6 & 601 & 8 & 797 & 0 & 997 & 2 \\
			 89 & 0 & 241 & 2 & 421 & 8 & 607 & 2 & 809 & 0 & 1009 & 2 \\
			 97 & 2 & 251 & 0 & 431 & 0 & 613 & 2 & 811 & 2 & 1013 & 0 \\
			 101 & 0 & 257 & 0 & 433 & 2 & 617 & 0 & 821 & 0 & 1019 & 0 \\
			 103 & 2 & 263 & 0 & 439 & 2 & 619 & 8 & 823 & 2 & 1021 & 2 \\
			 107 & 0 & 269 & 0 & 443 & 6 & 631 & 2 & 827 & 0 & 1031 & 0 \\
			 109 & 2 & 271 & 2 & 449 & 0 & 641 & 0 & 829 & 2 & 1033 & 2 \\
			 113 & 0 & 277 & 2 & 457 & 8 & 643 & 2 & 839 & 0 & 1039 & 8 \\
		\hline
		\end{array}
		\end{equation*}
		\caption{Values of $F(p) = F(p;1,1,1)$ as $p$ ranges over the first 174 odd primes.  Primes $p$ for which
		$F(p)=0$ satisfy the property that the corresponding Fermat equation $x^p + y^p = z^p$ has no solutions
		in integers with $p \nmid xyz$.}
		\label{TableFermat}
	\end{table}	

%% file: s/SectionHB.tex
\section{The fourth moment}\label{SectionHB}
	Using the fact that the matrices $T_i$ are simultaneously unitarily
	diagonalizable, we obtain a variety
	of quartic formulas involving Heilbronn sums.  
	
	\begin{Theorem}\label{TheoremQuartic}
		Letting $g$ denote a primitive root modulo $p^2$, for $1 \leq i,j,k,\ell \leq p$ we have
		\begin{align*}
			 p^2\sum_{r=1}^p c_{i,k,r} c_{j,\ell,r} &= 
			\sum_{r=1}^p H_p(g^{i+r}) H_p(g^{j+r}) H_p(g^{k+r}) H_p(g^{\ell+r}) \\
			&\qquad\qquad + 
			\begin{cases}
				- 2p^2 + 3p  & \text{if $i = k$ and $j = \ell$}, \\
				p^3- 4 p^2 + 3 p   & \text{if $i \neq k$ and $j \neq \ell$}, \\
				0 & \text{otherwise}.
			\end{cases}
		\end{align*}
		In particular,
		\begin{equation}\label{eq-Quartic}
			\sum_{\ell=1}^p H_p^4(g^{\ell}) =  p^2 \sum_{\ell=1}^p c_{i,i,\ell}^2 + 2p^2 - 3p.
		\end{equation}
	\end{Theorem}

	\begin{proof}
		Since $U = U^*$ it follows from Lemma \ref{LemmaSESUP} that $T_i T_j = UD_i D_j U$
		for $1 \leq i,j \leq p+2$.  
		Letting $1 \leq i,j,k,\ell \leq p$ and recalling that $|X_i| = p-1$ in this range, 
		it follows from the symmetry of $T_j$, \eqref{eq-cijp1}, and \eqref{eq-ciip2} that
		\begin{align*}
			[T_iT_j]_{k,\ell} 
			&= \sum_{r=1}^{p+2} \frac{ c_{i,k,r} \sqrt{ |X_r| } }{ \sqrt{ |X_k|} } \cdot \frac{ c_{j,r,\ell} \sqrt{ |X_\ell| } }{ \sqrt{ |X_r|} } \\
			&= \sum_{r=1}^{p+2} \frac{ c_{i,k,r} \sqrt{ |X_r| } }{ \sqrt{ |X_k|} } \cdot \frac{ c_{j,\ell,r} \sqrt{ |X_r| } }{ \sqrt{ |X_{\ell}|} } \\
			&= \frac{1}{p-1} \sum_{r=1}^{p+2} |X_r| c_{i,k,r} c_{j,\ell,r} \\
			&= \sum_{r=1}^{p} c_{i,k,r} c_{j,\ell,r} + 
			\begin{cases}
				p-1  & \text{if $i = k$ and $j = \ell$}, \\
				1 & \text{if $i \neq k$ and $j \neq \ell$}, \\
				0 & \text{otherwise}.
			\end{cases}
		\end{align*}
		On the other hand, using the fact that $U = U^T$ we have
		\begin{align*}
			[UD_iD_jU]_{k,\ell}
			&= \frac{1}{p^2} \sum_{r=1}^{p+2} \frac{\sigma_k(X_r) \sqrt{|X_r|}}{\sqrt{|X_k|}}\cdot\sigma_i(X_r)\sigma_j(X_r) 
				\cdot \frac{ \sigma_r(X_{\ell}) \sqrt{|X_{\ell}|}}{\sqrt{|X_r|}} \\
			&= \frac{1}{p^2} \sum_{r=1}^{p+2} \frac{|X_r| \sigma_i(X_r)\sigma_j(X_r)\sigma_k(X_r) \sigma_{\ell}(X_r) }{p-1}\\
			&= \frac{1}{p^2}\Big( \sum_{r=1}^p \sigma_i(X_r)\sigma_j(X_r)\sigma_k(X_r) \sigma_{\ell}(X_r)
				+ 1 + (p-1)^3\Big) \\
			&= \frac{1}{p^2}\Big( \sum_{r=1}^p H_p(g^{i+r}) H_p(g^{j+r}) H_p(g^{k+r}) H_p(g^{\ell+r}) + 1 + (p-1)^3\Big).
		\end{align*}
		Equating our expressions for the matrix entries $[T_iT_j]_{k,\ell}$ and $[UD_iD_jU]_{k,\ell}$
		yields the desired result.
	\end{proof}

	\begin{Corollary}\label{TheoremHB}
		If $1 \leq i\leq p$, then
		\begin{equation}\label{eq-Threshest}
			\max_{1 \leq k \leq p} c_{i,i,k} \ll p^\beta
			\quad \implies \quad H_p(u)\ll p^{\frac{3+\beta}{4}}
		\end{equation}
		whenever $p \nmid u$.
	\end{Corollary}

	\begin{proof}
	As a consequence of \eqref{eq-Quartic}, for $p \nmid u$ we obtain
	\begin{equation}\label{eq-RSEQ}
		H_p(u)\ll p^{\frac{1}{2}}\Big(\sum_{k=1}^{p} c_{i,i,k}^2\Big)^{\frac{1}{4}},
	\end{equation}
	thereby recovering Heath-Brown's observation \cite[Lem.~1]{HB2}.
	Lemma \ref{LemmaSRS} and the assumption that $c_{i,i,k} \ll p^{\beta}$ ensure that
            \begin{equation*}
                \sum_{k=1}^p c_{i,i,k}^2 \ll \sum_{k=1}^p p^{\beta} c_{i,i,k} = p^{\beta}(p-2) \leq p^{1 + \beta}.\qedhere
            \end{equation*}
		Plugging this into \eqref{eq-RSEQ} we obtain \eqref{eq-Threshest}. \qedhere
	\end{proof}

	A result of Mit'kin implies that we may take $\beta = \frac{2}{3}$ in \eqref{eq-Threshest} \cite{Mitkin},
	which yields Heath-Brown's estimate $H_p(u) \ll p^{11/12}$ for $p\nmid u$ \cite{HB1}.
	The upper bound was later improved by Heath-Brown and Konyagin
	to $p^{7/8}$ \cite{HBK} and by Shkredov to 
	$p^{59/68} \log^{5/34} p$~\cite{shkredov2012}.